\documentclass[12pt,a4paper]{amsart}

\usepackage[utf8]{inputenc}
\usepackage{enumitem,a4,titletoc}
\usepackage{bbm,amssymb,mathtools}
\usepackage[usenames,dvipsnames]{xcolor}
\usepackage[pagebackref,colorlinks,linkcolor=BrickRed,citecolor=OliveGreen,urlcolor=black,hypertexnames=true]{hyperref}
\usepackage[margin=2.5cm]{geometry}

\DeclareMathOperator{\GL}{GL}

\DeclareMathOperator{\dom}{dom}

\DeclareMathOperator{\edom}{edom}
\DeclareMathOperator{\edomst}{edom^{st}}
\DeclareMathOperator{\edomstn}{\edom^{\operatorname{st}}_n}

\DeclareMathOperator{\opm}{M}

\newcommand{\N}{\mathbb{N}}
\newcommand{\R}{\mathbb{R}}

\newcommand{\cD}{\mathcal{D}}

\newcommand{\cL}{\mathcal{L}}

\newcommand{\cR}{\mathcal{R}}

\newcommand{\cV}{\mathcal{V}}

\newcommand{\bb}{\mathbf{b}}
\newcommand{\cc}{\mathbf{c}}
\newcommand{\uu}{\mathbf{u}}
\newcommand{\vv}{\mathbf{v}}
\newcommand{\ww}{\mathbf{w}}

\newcommand{\rr}{\mathbbm r}
\newcommand{\mm}{\mathbbm m}

\newcommand{\rs}{\mathbbm s}
\newcommand{\kk}{\mathbbm k}

\newcommand{\ulx}{\boldsymbol{x}}

\newcommand{\ulxx}{\boldsymbol{\xi}}

\newcommand*{\mat}[1]{\opm_{#1}(\kk)}

\newcommand{\Langle}{\mathop{<}\!}
\newcommand{\Rangle}{\!\mathop{>}}
\newcommand{\mx}{\Langle \ulx\Rangle}
\newcommand{\px}{\kk\!\Langle \ulx\Rangle}

\makeatletter
\def\moverlay{\mathpalette\mov@rlay}
\def\mov@rlay#1#2{\leavevmode\vtop{
		\baselineskip\z@skip \lineskiplimit-\maxdimen
		\ialign{\hfil$#1##$\hfil\cr#2\crcr}}}
\makeatother

\newcommand{\re}{\cR_{\kk}(\ulx)}
\newcommand{\plangle}{\moverlay{(\cr<}}
\newcommand{\prangle}{\moverlay{)\cr>}}
\newcommand{\rx}{\kk\plangle \ulx \prangle}

\newtheorem{thm}{Theorem}[section]
\newtheorem{lem}[thm]{Lemma}
\newtheorem{cor}[thm]{Corollary}
\newtheorem{prop}[thm]{Proposition}

\theoremstyle{definition}
\newtheorem{defn}[thm]{Definition}
\newtheorem{exa}[thm]{Example}

\theoremstyle{remark}
\newtheorem{rem}[thm]{Remark}

\numberwithin{equation}{section}

\begin{document}
	
\setcounter{tocdepth}{3}
\contentsmargin{2.55em} 
\dottedcontents{section}[3.8em]{}{2.3em}{.4pc} 
\dottedcontents{subsection}[6.1em]{}{3.2em}{.4pc}
\dottedcontents{subsubsection}[8.4em]{}{4.1em}{.4pc}

\setcounter{page}{1}

\title[On domains of nc rational functions]{On domains of noncommutative rational functions}

\author[J. Vol\v{c}i\v{c}]{Jurij Vol\v{c}i\v{c}${}^\star$}
\address{Jurij Vol\v{c}i\v{c}, Department of Mathematics, The University of Auckland}
\email{jurij.volcic@auckland.ac.nz}
\thanks{${}^\star$Supported by The University of Auckland Doctoral Scholarship.}

\date{\today}


\begin{abstract}
In this paper the stable extended domain of a noncommutative rational function is introduced and it is shown that it can be completely described by a monic linear pencil from the minimal realization of the function. This result amends the singularities theorem of Kalyuzhnyi-Verbovetskyi and Vinnikov. Furthermore, for noncommutative rational functions which are regular at a scalar point it is proved that their domains and stable extended domains coincide.
\end{abstract}


\maketitle


\section{Introduction}

Noncommutative rational fractions are the elements of the universal skew field of a free algebra \cite{Coh}. While this skew field can be constructed in various ways \cite{Ami,Le,Li}, it can also be defined through evaluations of formal rational expressions on tuples of matrices \cite{KVV1}. This interpretation gives rise to prominent applications of noncommutative rational functions in free analysis \cite{KVV2,AM}, free real algebraic geometry \cite{HMV,OHMP,BPT,HM} and control theory \cite{BGM,BH}. The consideration of matrix evaluations naturally leads to the introduction of the {\it domain} of a noncommutative rational function. However, at first sight this notion seems intangible: since a noncommutative rational function is an equivalence class of formal rational expressions, its domain is defined as the union of the formal domains of all its representatives; see Subsection \ref{ss:21} for precise definition. Therefore new variants of domains emerged: {\it extended domains} \cite{KVV,KVV1} and {\it analytic} or {\it limit domains} \cite{HMV,HM}. Both of these notions are related to generic evaluations of noncommutative rational functions and can thus be described using a single representative of a function.

The main important breakthrough in characterizing domains was done by Dmitry Kaliuzhnyi-Verbovetskyi and Victor Vinnikov in \cite{KVV}. Perceptively combining linear systems realizations from control theory and difference-differential operators from free analysis they seemingly proved that the extended domain of a noncommutative rational function $\rr$ that is regular at the origin coincides with the invertibility set of the monic linear pencil from a minimal realization of $\rr$; see \cite[Theorem 3.1]{KVV}. While minimal realizations of noncommutative rational functions can be effectively computed \cite{BGM,BR}, monic linear pencils are key tools in matrix theory and are well-explored through control theory \cite{BEFB}, algebraic geometry \cite{Dol} and optimization \cite{WSV}. Therefore the result of Kaliuzhnyi-Verbovetskyi and Vinnikov proved to be of great importance in free real algebraic geometry and free function theory \cite{HMPV,BH,MS,HM,KV,KPV}. Alas, there is a gap in its proof and its conclusion does not hold. The reason behind is the fact that the extended domain of a noncommutative rational function is in general not closed under direct sums; a concrete instance when this occurs is given in Example \ref{counterex}.

The main results of this paper adjust and improve \cite[Theorem 3.1]{KVV}. First we recall domains and extended domains in Subsection \ref{ss:21} and give the necessary facts about realizations in Subsection \ref{ss:22}. Then we define the {\it stable extended domain} of a noncommutative rational function (Definition \ref{d:sted}). In Proposition \ref{p:dir} it is shown that the stable extended domain is always closed under direct sums, which is in contrast with the extended domain. In Theorem \ref{t:stdom} we prove that the following variant of \cite[Theorem 3.1]{KVV} holds.

{\it If a noncommutative rational function $\rr$ is regular at the origin, then its stable extended domain is equal to the invertibility set of the monic linear pencil from the minimal realization of $\rr$.}

This statement is then extended to noncommutative rational functions that are regular at some scalar point in Corollary \ref{c:shift}. Moreover, for such functions we are able to completely describe their domains due to the following result.

{\it If a noncommutative rational function is regular at some scalar point, then its domain and its stable extended domain coincide.}

See Theorem \ref{t:eq} for the proof. Finally, in Example \ref{counterex2} we present a noncommutative rational function whose domain is strictly larger than the domain of any of its representatives.

\subsection*{Acknowledgments} The author thanks Dmitry Kaliuzhnyi-Verbovetskyi and Victor Vinnikov for their consent to extend their work, Christoph Hanselka for his help with Lemma \ref{l:ch}, and Igor Klep for advice and comments.

\section{Domains, extended domains and linear pencils}

In this section we first recall the definition of noncommutative rational functions, their (extended) domains and the basics of realization theory. Then we present a counterexample to \cite[Theorem 3.1]{KVV} and explain why this phenomenon occurs.

\subsection{Skew field of noncommutative rational functions}\label{ss:21}

Throughout the paper let $\kk$ be a field of characteristic $0$. Let $\ulx=(x_1,\dots,x_g)$ be a tuple of freely noncommuting variables. By $\mx$ and $\px$ we denote the free monoid and the free unital $\kk$-algebra, respectively, generated by $\ulx$. Elements of $\mx$ and $\px$ are called {\it words} and {\it noncommutative (nc) polynomials}, respectively.

We introduce noncommutative rational functions using matrix evaluations of formal rational expressions following \cite{HMV,KVV1}. For their ring-theoretic origins see \cite{Ami,Coh}. {\bf Noncommutative (nc) rational expressions} are syntactically valid combinations of elements in $\kk$, variables in $\ulx$, arithmetic operations $+,\cdot,{}^{-1}$ and parentheses $(,)$. For example, $(x_3+x_2x_3^{-1}x_1)^{-1}+1$, $x_1^{-1}x_2^{-1}-(x_2x_1)^{-1}$ and $(1-x_1^{-1}x_1)^{-1}$ are nc rational expressions.

Fix $n\in\N$. Given a nc rational expression $r$ and $X\in \mat{n}^g$, the evaluation $r(X)$ is defined in the obvious way if all inverses appearing in $r$ exist at $X$.  Let $\dom_n r$ be the set of all $X\in\mat{n}^g$ such that $r$ is regular at $X$. Note that $\dom_n r$ is Zariski open in $\mat{n}^g$. The set
$$\dom r=\bigcup_{n\in\N}\dom_n r$$
is called {\bf the domain of $r$}. We say that a nc rational expression $r$ is {\bf non-degenerate} if $\dom r\neq\emptyset$. Let $\re$ denote the set of all non-degenerate expressions and on it we define an equivalence relation $r_1\sim r_2$ if and only if $r_1(X)=r_2(X)$ for all $X\in\dom r_1\cap \dom r_2$. The equivalence classes with respect to this relation are called {\bf noncommutative (nc) rational functions}. By \cite[Proposition 2.1]{KVV1} they form a skew field denoted $\rx$, which is the universal skew field of fractions of $\px$; see \cite[Section 4.5]{Coh} for an exposition on universal skew fields. The equivalence class of a nc rational expression $r\in\re$ is written as $\rr\in\rx$. Denote
$$\dom_n\rr=\bigcup_{r\in\rr}\dom_n r,\qquad \dom \rr=\bigcup_{n\in\N}\dom_n \rr$$
and call $\dom\rr$ {\bf the domain of $\rr$}.

Again fix $n\in\N$ and let $\Xi=(\Xi_1,\dots,\Xi_g)$ be the $n\times n$ {\bf generic matrices}, i.e., the matrices whose entries are independent commuting variables. If $\dom_n r\neq \emptyset$ for $r\in\re$, then $r[n]:=r(\Xi)$ is a $n\times n$ matrix of commutative rational functions in $gn^2$ variables. If $r_1,r_2\in\re$ and $\dom_n r_1\neq\emptyset\neq\dom _n r_2$, then $r_1\sim r_2$ clearly implies $r_1[n]=r_2[n]$. For $\rr\in\rx$ with $\dom_n\rr\neq\emptyset$ we can thus set $\rr[n]:=r[n]$ for arbitrary $r\in\rr$ with $\dom_n r\neq\emptyset$. Since the ring of commutative polynomials is a unique factorization ring, every commutative rational function is a quotient of two (up to a scalar multiple) unique coprime polynomials and therefore has a well-defined domain, namely the complement of the zero set of its denominator. Let $\edom_n \rr\subseteq \mat{n}^g$ be the intersection of the domains of entries in $\rr[n]$ if $\dom_n\rr\neq\emptyset$; otherwise set $\edom_n\rr=\emptyset$. Note that $\edom_n\rr$ is Zariski open in $\mat{n}^g$. The set
$$\edom \rr=\bigcup_{n\in\N}\edom_n \rr$$
is called {\bf the extended domain of $\rr$}; see \cite[Subsection 2.1]{KVV}.

\subsection{Realization theory}\label{ss:22}
 
A powerful tool for operating with nc rational functions regular at the origin are realizations, coming from control theory \cite{BGM} and automata theory \cite{BR}. Let $\rr\in\rx$ and assume that $0\in\dom\rr$. Then there exist $d\in\N$, vectors $\bb,\cc\in\kk^{d\times 1}$ and a {\bf monic linear pencil} $L=I-\sum_jA_jx_j$ with $A_1,\dots,A_g\in\kk^{d\times d}$ such that
\begin{equation}\label{e:real0}\rr=\cc^t L^{-1}\bb\end{equation}
holds in $\rx$. The tuple $(\cc,L,\bb)$ is called a {\bf (recognizable series) realization} of $\rr$ of size $d$; cf. \cite{BR}. In general, $\rr$ admits various realizations. The ones whose size is minimal among all realizations of $\rr$ are called {\bf minimal}. It is well known (see e.g. \cite[Theorem 2.4]{BR}) that minimal realizations are unique up to similarity: if  $(\cc,L,\bb)$ and $(\cc',L',\bb')$ are minimal realizations of $\rr$ of size $d$, then there exists $P\in\GL_d(\kk)$ such that $\cc'=P^{-t}\cc$, $\bb'=P\bb$ and $A_j'=PA_jP^{-1}$ for $1\le j\le g$.

For a monic pencil $L$ as above and $X\in\mat{n}^g$ we write $L(X)=I\otimes I-\sum_jA_j\otimes X_j$. For $\rr\in\rx$ with $0\in\dom\rr$ define
$$\cD_n(\rr;0)=\left\{X\in\mat{n}^g\colon \det L(X)\neq0\right\},\qquad \cD(\rr;0)=\bigcup_n \cD_n(\rr;0),$$
where $(\cc,L,\bb)$ is a minimal realization of $\rr$. Note that $\cD(\rr;0)$ is independent of the choice of a minimal realization because of its uniqueness up to similarity. This set is also called the {\it invertibility set} of $L$ \cite[Subsection 1.1]{HM}.

More generally, if $\alpha\in\dom_1\rr$, then $\rr_\alpha(x)=\rr(x+\alpha)$ is a nc rational function regular at $0$ and we define
$$\cD_n(\rr;\alpha)=I_n\alpha+\cD_n(\rr_\alpha;0),\qquad \cD(\rr;\alpha)=\bigcup_n \cD_n(\rr;\alpha).$$

\subsection{Ill-behavior of the extended domain}

Obviously we have $\dom\rr\subseteq\edom\rr$. Also, if $\alpha\in\dom_1\rr$, then $\cD(\rr;\alpha)\subseteq\edom\rr$. In the following example we show that $\cD(\rr;0)$ is not necessarily equal to $\edom\rr$, thus presenting a counterexample to \cite[Theorem 3.1]{KVV}.

\begin{exa}\label{counterex}
Let $g=2$ and $r=(1-x_1)x_2(1-x_1)^{-1}$. Then $\rr$ admits a minimal realization
$$\begin{pmatrix}0 & 1 & 0\end{pmatrix}
\begin{pmatrix}
 1 & 0 & -x_2 \\
 x_1 & 1 & -x_2 \\
 0 & 0 & 1-x_1
\end{pmatrix}^{-1}
\begin{pmatrix}0 \\ 0 \\ 1 \end{pmatrix},$$
so $D_1(\rr;0)=\{(x_1,x_2)\colon x_1\neq 1\}$. On the other hand, if $\xi_1,\xi_2$ are commuting independent variables, then $(1-\xi_1)\xi_2(1-\xi_1)^{-1}=\xi_2$, so $\edom_1\rr=\kk\times\kk$. Therefore $\cD(\rr;0)\neq \edom\rr$.
\end{exa}

\begin{rem}\label{r:ex1}
\cite[Theorem 3.1]{KVV} fails for $r=(1-x_1)x_2(1-x_1)^{-1}$ because \cite[Corollary 2.20]{KVV} does not hold. Let $\cL_2$ be the left shift operator on nc rational functions regular at the origin with respect to $x_2$; see \cite[Subsection 2.2]{KVV} for the definition. Then $\cL_2(r)=(1-x_1)^{-1}$ and hence $\edom \cL_2(\rr)\nsupseteq \edom \rr$, contradicting \cite[Corollary 2.20]{KVV}. Furthermore, this lack of inclusion occurs because $\edom \rr$ is not closed under direct sums: for example, $(1,1),(0,0)\in\edom_1 \rr$ but $(1\oplus 0,1\oplus 0)\notin\edom_2 \rr$ and $(I_2,I_2)\notin\edom_2 \rr$.
\end{rem}

\begin{rem}
While the sets $\dom\rr$ and $\edom \rr$ are both closed under simultaneous conjugation, we have seen in Remark \ref{r:ex1} that $\edom\rr$ is not closed under direct sums. In terms of noncommutative function theory \cite{KVV2}, the extended domain of a nc rational function $\rr$ is thus not necessarily a noncommutative set and is therefore not a ``natural'' domain of the noncommutative function $\rr$. 

Moreover, since $\edom\rr\neq \cD(\rr;0)$ in general, it is a priori not even clear that $\dom\rr\subseteq D(\rr;0)$. This could encumber operating with the domain of a nc rational function. Note that since the domain of a function is defined as the union of the domains of its representatives, determining whether a given tuple of matrices belong to the domain of $\rr$ is not straightforward. In particular, we would not even know if $\dom\rr$ is closed under direct sums, which could further compromise the study of nc rational functions from the free analysis perspective (note however that $\dom r$ is closed under direct sums for every $r\in\re$).
\end{rem}

\begin{rem}\label{r:ampl}
If $X\in\mat{n}^g$ and $\ell\in\N$, then
$$I_\ell\otimes X=\overbrace{X\oplus\cdots\oplus X}^\ell\in\mat{\ell n}^g$$
is an {\it ampliation} of $X$. As seen in Remark \ref{r:ex1}, $\edom\rr$ is in general not even closed under ampliations. However, observe that $I_\ell\otimes X\in\edom\rr$ for some $\ell\in\N$ implies $X\in\edom\rr$.
\end{rem}

\section{Main results}

In this section we introduce the notion of a stable extended domain of a nc rational function. Unlike the extended domain, the stable extended domain is closed under direct sums and Theorem \ref{t:stdom} shows that it can be described by a monic linear pencil. In Theorem \ref{t:eq} we furthermore show that the stable extended domain and the domain coincide for nc rational functions regular at some scalar point.

\subsection{Stable extended domain}

We start by defining the notion of a domain that will help us mend \cite[Theorem 3.1]{KVV}.

\begin{defn}\label{d:sted}
Let $\rr\in\rx$. {\bf The stable extended domain} of $\rr$ is
$$\edomst \rr=\bigcup_{n\in\N} \edomstn \rr,$$
where
$$\edomstn \rr=\left\{X\in\edom_n \rr\colon I_\ell\otimes X\in\edom\rr \text{ for all } \ell\in\N\right\}.$$
\end{defn}

Observe that $\edomst\rr$ is closed under simultaneous conjugation and
$$\dom\rr\subseteq\edomst\rr\subseteq\edom\rr.$$

\begin{lem}\label{l:tech}
Let $\rr\in\rx$ and assume $\dom_n\rr\neq\emptyset$. Let $\Xi$ be a tuple of $(2n)\times (2n)$ generic matrices and let $p=p(\Xi)$ be the least common multiple of the denominators of entries in $\rr[2n]$. If $\Xi'$ and $\Xi''$ are tuples of independent $n\times n$ generic matrices, then
$$p(\Xi'\oplus \Xi'')=p_1(\Xi')p_2(\Xi'')$$
for some commutative polynomials $p_1$ and $p_2$.
\end{lem}

\begin{proof}
Let $r$ be a representative of $\rr$ with $\dom_n r\neq\emptyset$. By \cite[Theorem 5.2]{HMS} there exist $d\in\N$, $\uu,\vv\in\kk^{d\times 1}$ and $Q=Q_0+\sum_{j=1}^gQ_jx_j$ with $Q_j\in\kk^{d\times d}$ such that
$$\rr=\vv^tQ^{-1}\uu$$
holds in $\rx$ and $Q(X)$ is invertible for every $X\in\dom r$. Since $\dom_n r\neq\emptyset$, $q=\det Q(\Xi)$ is a nonzero polynomial. If $a$ is an invertible square matrix, then the entries of $(\det a)a^{-1}$ are polynomials in entries of $a$. Hence we conclude that $p$ divides $q$. Observe that
$$q(\Xi'\oplus \Xi'')=\det Q(\Xi'\oplus \Xi'')=\det Q(\Xi')\det Q(\Xi'')=q_1(\Xi') q_2(\Xi'').$$
Since $p(\Xi'\oplus \Xi'')$ divides $q_1q_2$, which is a product of polynomials in disjoint sets of variables, we see that $p(\Xi'\oplus \Xi'')=p_1p_2$ for some $p_1$ dividing $q_1$ and $p_2$ dividing $q_2$.
\end{proof}

\begin{prop}\label{p:dir}
Every stable extended domain is closed under direct sums.
\end{prop}

\begin{proof}
If $Y'\in \mat{n'}^g$ and $Y''\in \mat{n''}^g$, then $I_{n'n''}\otimes (I_\ell\otimes (Y'\oplus Y''))$ is after a canonical shuffle equal to $(I_{\ell n''}\otimes Y')\oplus (I_{\ell n'}\otimes Y'')$. 
By Remark \ref{r:ampl} and Definition \ref{d:sted} it therefore suffices to prove the following claim: if $X',X''\in\edomstn \rr$, then $X'\oplus X''\in\edom_{2n} \rr$.

Let $p$ be the least common multiple of the denominators of entries in $\rr[2n]$. Then $$\edom_{2n}\rr=\left\{X\in\mat{2n}^g\colon p(X)\neq0\right\}.$$
Let $\Xi'$ and $\Xi''$ be tuples of independent $n\times n$ generic matrices. By Lemma \ref{l:tech} we have
$$p(\Xi'\oplus \Xi'')=p_1(\Xi')p_2(\Xi'')$$
for some commutative polynomials $p_1$ and $p_2$. Hence
$$X'\oplus X''\in\edom_{2n} \rr \iff p_1(X')p_2(X'')\neq0.$$
Since $X'\oplus X',X''\oplus X''\in\edom_{2n} \rr$, we have 
$$p_1(X')p_2(X')\neq0,\qquad p_1(X'')p_2(X'')\neq0$$
and thus $p_1(X')p_2(X'')\neq0$, so $X'\oplus X''\in \edom_{2n} \rr$.
\end{proof}

\subsection{Corrigendum to the theorem of Kaliuzhnyi-Verbovetskyi and Vinnikov}

In this subsection we correct the proof and the statement of \cite[Theorem 3.1]{KVV}. The only gap in their ingenious proof appears at the very beginning in \cite[Remark 2.7]{KVV}. From the assumption that $\edom\rr$ is closed under direct sums they derive that the left and right shifts of $\rr$ do not shrink the domain in \cite[Corollary 2.20]{KVV}. Then they show in \cite[Theorem 3.1]{KVV} that if $(\cc,L,\bb)$ is a minimal realization of $\rr$, then every entry of $L$ is a linear combination of iterated left and right shifts of $\rr$. Combining these two results they deduce that $\edom\rr=\cD(\rr;0)$.  We will show that these two main steps are indeed valid if $\edom\rr$ is replaced by $\edomst\rr$.

Let $\cL_j$ be the left shift operator with respect to $x_j$ as defined in \cite[Subsection 2.2]{KVV}. By \cite[Remark 2.25]{KVV}, $\cL_j$ can be described as follows. Every $\rr\in\rx$ that is regular at the origin can be expanded into a noncommutative power series about $0$
$$\sum_{w\in\mx}\alpha_ww;$$
then $\cL_j(\rr)\in\rx$ is regular at $0$ with the power series expansion
$$\sum_{w\in\mx}\alpha_{x_jw}w.$$

\begin{lem}\label{l:shift}
If $0\in\dom \rr$, then $\edomst \rr\subseteq \edomst \cL_j(\rr)$ for $1\le j \le g$.
\end{lem}

\begin{proof}
By Definition \ref{d:sted} it suffices to prove $\edomst \rr\subseteq \edom \cL_j(\rr)$. If $X\in\edomstn\rr$, then $X\oplus 0\in\edom_{n+1}\rr$ by Proposition \ref{p:dir}. If $\Xi$ is a tuple of $n\times n$ generic matrices, $\ulxx=(\ulxx_1,\dots,\ulxx_g)$ a tuple of generic rows and $\zeta$ a commutative variable, then the properties of arithmetic operations on block lower triangular matrices imply that the denominators of the entries in
$$\rr[n+1]\begin{pmatrix}\Xi & 0\\ \ulxx & \zeta\end{pmatrix}$$
are independent of variables in $\ulxx$. Hence we have
$$\begin{pmatrix}X & 0\\ \ww & 0\end{pmatrix}\in\edom_{n+1}\rr$$
for all $\ww\in(\kk^{1\times n})^g$. Exactly as in the proof of \cite[Corollary 2.20]{KVV} we apply \cite[Theorem 2.19]{KVV} to show that $X$ then belongs to the intersection of the domains of the denominators of entries in the row vector $\sum_j \ww_j \cL_j(\rr)[n]$ for every $\ww\in(\kk^{1\times n})^g$. In particular, by choosing $\ww$ from the standard basis of $(\kk^{1\times n})^g$ we conclude that $X\in\edom \cL_j(\rr)$.
\end{proof}

\begin{thm}\label{t:stdom}
Let $\rr\in\rx$. If $0\in\dom \rr$, then $\edomst\rr= \cD(\rr;0)$.
\end{thm}

\begin{proof}
Observe that $\cD(\rr;0)\subseteq \edomst\rr$ because $\cD(\rr;0)$ is closed under direct sums. The converse inclusion is (mutatis mutandis) proved precisely in the same way as \cite[Theorem 3.1]{KVV}, where we utilize Lemma \ref{l:shift} instead of \cite[Corollary 2.20]{KVV}.
\end{proof}

\begin{rem}
As mentioned in \cite[Remark 3.2]{KVV}, the recognizable series realization implicitly appearing in Theorem \ref{t:stdom} can be in fact replaced by a more general form of a state space realization, which also covers the noncommutative Fornasini-Marchesini realization \cite{BGM}, the noncommutative Kaliuzhnyi-Verboverskyi realization \cite{BKV} and the pure butterfly realization \cite{HMV}.
\end{rem}

\begin{cor}\label{c:shift}
Let $\rr\in\rx$. If $\alpha\in\dom_1\rr$, then $\edomst\rr=\cD(\rr;\alpha)$.
\end{cor}

\begin{proof}
(1) Let $\rr_\alpha=\rr(x+\alpha)\in\rx$. Then
$$\edomstn\rr=I_n\alpha+\edomstn\rr_\alpha\qquad\text{and}\qquad\cD(\rr;\alpha)=I_n\alpha+\cD(\rr_\alpha;0)$$
for all $n\in\N$. Now we apply Theorem \ref{t:stdom} to $\rr_\alpha$ to yield the conclusion.
\end{proof}

\begin{rem}
In \cite{Vol}, matrix coefficient realization theory is applied to extend Corollary \ref{c:shift} to arbitrary nc rational functions (i.e., those not necessarily defined at a scalar point) in terms of their Sylvester realizations. Thus every stable extended domain can be described as the invertibility set of a generalized monic pencil; see \cite[Corollary 5.9]{Vol} for the precise statement.
\end{rem}

\subsection{Domains of functions regular at a scalar point}

In this subsection we improve Corollary \ref{c:shift} and precisely describe $\dom \rr$ for a nc rational function $\rr$ with $\dom_1\rr\neq\emptyset$. We require the following technical lemma.

\begin{lem}\label{l:ch}
Let $\mm$ be a $d\times d$ matrix over $\rx$ and let $X\in\dom \mm$. If $\det \mm(X)\neq0$, then there exist nc rational expressions $s_{ij}$ such that $X\in\dom s_{ij}$ and $\mm^{-1}=(\rs_{ij})_{ij}$.
\end{lem}

\begin{proof}
We prove the statement by induction on $d$. If $d=1$, then $X\in\dom\mm$ implies that there exists an expression $m\in\mm$ with $X\in\dom m$. Since $\mm(X)=m(X)$ is invertible, $m^{-1}$ is the desired expression.

Now assume that the statement holds for $d-1$. Let $m_{ij}$ be rational expressions satisfying $m_{ij}\in\mm_{ij}$ and $X\in\dom m_{ij}$. Since $\mm(X)$ is invertible, there exists a univariate polynomial $f\in\kk[t]$ such that $f(\mm(X))\mm(X)=I$. Write $\widetilde{\mm}=f(\mm)\mm$. Let $u\in\widetilde{\mm}_{11}$ be such that $X\in\dom u$. Since $\widetilde{\mm}(X)=I$, we have $u(X)=I$ and consequently $X\in\dom u^{-1}$. As $\widetilde{\mm}_{11}$ is a nonzero nc rational function, the Schur complement $\widehat{\mm}$ of $\widetilde{\mm}$ with respect to $\widetilde{\mm}_{11}$ is well-defined. Note that $\widehat{\mm}$ is a $(d-1)\times(d-1)$ matrix whose entries are products and sums of expressions $m_{ij},u^{-1}$. Since $\widetilde{\mm}(X)$ is invertible, $\widehat{\mm}(X)$ is also invertible, so we can apply the induction hypothesis to $\widehat{\mm}$. Hence there exist rational expressions $s_{ij}'$ such that $X\in\dom s_{ij}'$ and $\widehat{\mm}^{-1}=(\rs_{ij}')_{ij}$. The entries of $\widetilde{\mm}^{-1}$ can be represented by sums and products of expressions $m_{ij},u^{-1},s_{ij}'$, hence the same holds for $\mm^{-1}=\widetilde{\mm}^{-1}f(\mm)$.
\end{proof}

\begin{thm}\label{t:eq}
If $\dom_1 \rr\neq\emptyset$, then $\dom\rr= \edomst\rr=\cD(\rr;\alpha)$ for every $\alpha\in\dom_1\rr$.
\end{thm}

\begin{proof}
By Corollary \ref{c:shift} we have $\dom\rr\subseteq \edomst\rr=\cD(\rr;\alpha)$. On the other hand, let $\cc^t L^{-1}\bb$ be a minimal realization of $\rr$ about $\alpha$ and $X\in\cD(\rr;\alpha)$. Hence $\det L(X)\neq0$, so by Lemma \ref{l:ch} there exist nc rational expressions $s_{ij}$ such that $X\in\dom s_{ij}$ and $L^{-1}=(\rs_{ij})_{ij}$. Therefore $\rr$ admits a representative
$$r=\sum_{ij} c_is_{ij}b_j$$
with $X\in\dom r\subseteq\dom\rr$.
\end{proof}

\begin{rem}
Theorem \ref{t:eq} can be compared with \cite{HM}, where authors consider {\it limit domains} and {\it hidden singularities} of symmetric nc rational functions over $\R$ that are regular at the origin. While the limit domain of a nc rational function might strictly contain its extended domain, we still observe the following analogy. If $\rr\in\rx$ and $0\in\dom\rr$, then \cite[Theorem 1.5]{HM} roughly corresponds to $\dom\rr\cap(\edom\rr\setminus\cD(\rr;0))=\emptyset$, i.e., $\dom\rr\subseteq\cD(\rr;0)$.
\end{rem}

\begin{rem}\label{r:mncre}
Nc rational expressions as defined in Subsection \ref{ss:21} are sometimes also called {\it scalar} nc rational expressions. On the other hand, a {\it matrix nc rational expression} is a syntactically valid combination of matrices over scalar nc rational expressions and matrix operations $+,\cdot,{}^{-1}$. Here we assume that the sum and the product are applied just to matrices of matching dimensions and the inverse is applied just to square matrices. Hence a nc rational function can be also represented by a $1\times 1$ matrix nc rational expression. For example, if $\dom_1\rr\neq\emptyset$, then any realization of $\rr$ can be viewed as a $1\times 1$ matrix representative of $\rr$. Now Lemma \ref{l:ch} implies that any point $X\in\mat{n}^g$, at which some $1\times 1$ matrix representative of $\rr$ is defined, belongs to $\dom\rr$ (which is defined only with scalar representatives).
\end{rem}

\subsection{Absence of distinguished representatives}

Recall that the domain of a nc rational function is defined as the union of domains of its representatives. If $\rr\in\rx$ and $r\in\rr$, then we say that $r$ is a {\it distinguished representative} of $\rr$ if $\dom\rr=\dom r$. For example, if $\rr$ is the inverse of $f\in\px\setminus\{0\}$, then $\dom\rr=\edom\rr=\dom f^{-1}$, where $f^{-1}$ is naturally considered as an element of $\re$. While the existence of distinguished representatives may seem exceptional, up to this point there was no reported instance of a function without distinguished representatives. We present such a function in the following example.

\begin{exa}\label{counterex2}
Let $g=4$ and
$$\rr=\left(x_4-x_3x_1^{-1}x_2\right)^{-1}\in \rx,\qquad
m=\begin{pmatrix}x_1 & x_2 \\ x_3& x_4\end{pmatrix} \in \px{\!}^{2\times 2}.$$
We claim that $\rr$ does not admit distinguished representatives. Since $(1,0,0,1)\in\dom_1 \rr$, we can use the minimal realization
$$\begin{pmatrix}0 & 1 \end{pmatrix}
\begin{pmatrix}1+x_1 & x_2 \\ x_3 & 1+x_4\end{pmatrix}^{-1}
\begin{pmatrix}0 \\ 1\end{pmatrix}
$$
for $\rr(x_1+1,x_2,x_3,x_4+1)$ and Theorem \ref{t:eq} to show that
$$\dom \rr=\left\{X\colon \det m(X) \neq0\right\}.$$
Suppose there exist $r\in\rr$ with $\dom r=\dom\rr$. Then $r$ contains at least one inverse. Looking at the innermost nested inverse we conclude that there exists $f\in\px$ of degree $d>0$ such that
\begin{equation}\label{e:fm}
\det f(X)=0 \quad\Rightarrow\quad \det m(X) =0\qquad \forall X\in\mat{n}^4,\ n\in\N.
\end{equation}
Note that $f(x_2,x_1,x_4,x_3)$, $f(x_3,x_4,x_1,x_2)$ and $f(x_4,x_3,x_2,x_1)$ also satisfy \eqref{e:fm}, so we can assume that $f$ contains a monomial $x_1u_0$ with $u_0\in\mx$ and $|u_0|=d-1$. Let $u_1,\dots,u_M\in\mx$ be all the words of length at most $d-1$ except $u_0$ and let $w_1,\dots,w_N\in\mx$ be all the words of length $d$ except $x_1u_0$; hence $M=\frac{4^d-1}{3}-1$ and $N=4^d-1$. Let $\cV\subset\px$ be a $\kk$-subspace spanned by $u_0,\dots,u_M,w_1,\dots,w_N$. Then $f=x_1u_0-h$ for $h\in \cV$. Furthermore, for $k\in\{2,3\}$ choose some ordering of
$$\{v^{(k)}_1,\dots, v^{(k)}_N\} = \{u_0,\dots,u_M,w_1,\dots,w_N\}\setminus x_k\mx.$$
For $k\in\{1,2,3,4\}$ let $X_k:\cV\to \cV$ be linear maps defined by
\begin{align*}
X_k u_0 &= \left\{\begin{array}{lr}
h & k=1 \\
x_ku_0 & k>1
\end{array}\right.,\\
X_k u_i &= x_ku_i \qquad 1\le i\le M,\\
X_k w_j &= \left\{\begin{array}{lr}
0 & k\in\{1,4\} \\
v^{(k)}_j & k\in\{2,3\}
\end{array}\right..
\end{align*}
Let $X=(X_1,X_2,X_3,X_4)$. By definition, we have
$$f(X)1=h-h=0.$$
On the other hand, if we identify $\cV$ with $\kk^{1+M+N}$ and the words $u_0,\dots,u_M,w_1,\dots,w_N$ with the standard basis vectors in $\kk^{1+M+N}$, then the columns of the block matrix
$$\begin{pmatrix} X_1 & X_2 \\ X_3 & X_4 \end{pmatrix}
\setlength\arraycolsep{0.24em}
=\left(\begin{array}{*{20}c}
h & x_1u_1 & \cdots & x_1u_M & 0 & \cdots & 0 & x_2u_0 & x_2u_1 & \cdots & x_2u_M & v^{(2)}_1 & \cdots & v^{(2)}_N \\
x_3u_0 & x_3u_1 & \cdots & x_3u_M & v^{(3)}_1 & \cdots & v^{(3)}_N & x_4u_0 & x_4u_1 & \cdots & x_4u_M & 0 & \cdots & 0
\end{array}\right)$$
are independent by the construction. Therefore $f(X)$ is singular and $m(X)$ is invertible, a contradiction.
\end{exa}

Note that, due to Remark \ref{r:mncre} and Theorem \ref{t:eq}, one can still say that a noncommutative rational function regular at a scalar point admits a distinguished $1\times 1$ matrix representative, namely its minimal realization.



\begin{thebibliography}{KK}

\bibitem[AM15]{AM}
J. Agler, J. E. McCarthy:
{\it Global holomorphic functions in several noncommuting variables}
Canad. J. Math. 67 (2015) 241--285.

\bibitem[Ami66]{Ami}
S.A. Amitsur:
{\it Rational identities and applications to algebra and geometry},
J. Algebra 3 (1966) 304--359.

\bibitem[BGM05]{BGM}
J. A. Ball, G. Groenewald, T. Malakorn:
{\it Structured noncommutative multidimensional linear systems}, 
SIAM J. Control Optim. 44 (2005) 1474--1528.

\bibitem[BEFB94]{BEFB}
S. Boyd, L. El Ghaoui, E. Feron, V. Balakrishnan:
{\it Linear matrix inequalities in system and control theory},
SIAM Studies in Applied Mathematics {15}, Society for Industrial and Applied Mathematics (SIAM), Philadelphia, PA, 1994.

\bibitem[BK-V08]{BKV}
J. A. Ball, D. Kaliuzhnyi-Verbovetskyi:
{\it Conservative dilations of dissipative multidimensional systems: the commutative and non-commutative settings}, 
Multidimens. Syst. Signal Process. 19 (2008) 79--122. 

\bibitem[BH10]{BH}
J. A. Ball, S. ter Horst:
{\it Robust control, multidimensional systems and multivariable Nevanlinna-Pick interpolation}, in
{\it Topics in Operator Theory}, 13--88, Birkhäuser Basel, 2010.

\bibitem[BPT13]{BPT}
G. Blekherman, P. A. Parrilo, R. R. Thomas (eds.):
{\it Semidefinite optimization and convex algebraic geometry},
MOS-SIAM Ser. Optim. 13, SIAM, Philadelphia, PA, 2013.

\bibitem[BR11]{BR}
J. Berstel, C. Reutenauer:
{\it Noncommutative rational series with applications},
Encyclopedia of Mathematics and its Applications 137, 
Cambridge University Press, Cambridge, 2011.

\bibitem[Coh95]{Coh}
P. M. Cohn:
{\it Skew fields. Theory of general division rings}, 
Encyclopedia of Mathematics and its Applications 57, Cambridge University Press, Cambridge, 1995.

\bibitem[Dol12]{Dol}
I. Dolgachev:
{\it Classical Algebraic Geometry: A Modern View},
Cambridge University Press, Cambridge, 2012.

\bibitem[HMS+]{HMS}
J. W. Helton, T. Mai, R. Speicher:
{\it Applications of realizations (aka linearizations) to free probability},
preprint \texttt{arXiv:1511.05330}.

\bibitem[HM14]{HM}
J. W. Helton, S. McCullough:
{\it Free convex sets defined by rational expressions have LMI representations},
J. Convex Anal. 21 (2014) 425--448.

\bibitem[HMPV09]{HMPV}
J. W. Helton, S. McCullough, M. Putinar, V. Vinnikov:
{\it Convex matrix inequalities versus linear matrix inequalities},
IEEE Transactions on Automatic Control  54 (2009) 952--964.

\bibitem[HMV06]{HMV}
J. W. Helton, S. McCullough, V. Vinnikov:
{\it Noncommutative convexity arises from linear matrix inequalities},
J. Funct. Anal. 240 (2006) 105--191.

\bibitem[K-VV09]{KVV}
D. S. Kaliuzhnyi-Verbovetskyi, V. Vinnikov:
{\it Singularities of rational functions and minimal factorizations: the noncommutative and the commutative setting}, 
Linear Algebra Appl. 430 (2009) 869--889.

\bibitem[K-VV12]{KVV1}
D. S. Kaliuzhnyi-Verbovetskyi, V. Vinnikov:
{\it Noncommutative rational functions, their difference-differential calculus and realizations}, 
Multidimens. Syst. Signal Process. 23 (2012) 49--77.

\bibitem[K-VV14]{KVV2}
D. S. Kalyuzhnyi-Verbovetskyi, V. Vinnikov:
{\it Foundations of free noncommutative function theory},
Mathematical Surveys and Monographs 199,
American Mathematical Society, Providence RI, 2014.

\bibitem[KV+]{KV}
I. Klep, J. Vol\v{c}i\v{c}:
{\it Free loci of matrix pencils and domains of noncommutative rational functions},
preprint 
\href{http://arxiv.org/abs/1512.02648}{\tt arXiv:1512.02648}.

\bibitem[KPV+]{KPV}
I. Klep, J. E. Pascoe, J. Vol\v{c}i\v{c}:
{\it Regular and positive noncommutative rational functions},
preprint 
\href{http://arxiv.org/abs/1605.03188}{\tt arXiv:1605.03188}.

\bibitem[Le74]{Le}
J. Lewin:
{\it Fields of fractions for group algebras of free groups},
Trans. Amer. Math. Soc. 192 (1974) 339--346.

\bibitem[Li00]{Li}
A. I. Lichtman:
{\it On universal fields of fractions for free algebras},
J. Algebra 231 (2000) 652--676.

\bibitem[MS13]{MS}
P. S. Muhly, B. Solel:
{\it Tensorial function theory: from Berezin transforms to Taylor’s Taylor series and back},
Integral Equations and Operator Theory 76 (2013) 463--508.

\bibitem[OHMP09]{OHMP}
M. C. de Oliveira, J. W. Helton, S. A. McCullough, M. Putinar:
{\it Engineering systems and free semi-algebraic geometry}, in
{\it Emerging applications of algebraic geometry}, 17--61, IMA Vol. Math. Appl., 149, Springer, New York, 2009.

\bibitem[Vol+]{Vol}
J. Vol\v{c}i\v{c}:
{\it Matrix coefficient realization theory of noncommutative rational functions},
preprint \texttt{arXiv:1505.07472}.

\bibitem[WSV12]{WSV}
H. Wolkowicz, R. Saigal, L. Vandenberghe (editors):
{\it Handbook of semidefinite programming: theory, algorithms, and applications},
vol. 27, Springer Science \& Business Media, 2012.

\end{thebibliography}
\end{document}